\documentclass[a4paper]{amsart}

\usepackage{amssymb,amsmath,amsthm,amsfonts,amscd}
\usepackage{setspace} 
\usepackage{color}
\usepackage{fancyhdr}
\usepackage{booktabs}
\usepackage{tabularx}
\usepackage{enumerate}
\usepackage{enumitem}
\usepackage{graphicx}
\usepackage{wrapfig}
\usepackage{mdframed}
\usepackage[final]{pdfpages}
\usepackage[retainorgcmds]{IEEEtrantools}

\newtheorem{thm}{Theorem}

\newtheorem{prop}[thm]{Proposition}

\newtheorem{corollary}[thm]{Corollary}
\newtheorem{Proposition}[thm]{Proposition}

\newtheorem{Definition}[thm]{Definition}

\newtheorem{Example}[thm]{Example}

\newtheorem{Remark}[thm]{Remark}

\newcommand{\C}{\mathbb{C}}

\newcommand{\K}{\mathbb{K}}
\newcommand{\Z}{\mathbb{Z}}

\newcommand{\cala}{\mathcal{A}}

\newcommand{\eps}{\varepsilon}

\newcommand{\ra}{\rightarrow}

\DeclareMathOperator{\Hom}{Hom}

\newcommand{\ip}[1]{\langle#1\rangle}

\newcommand{\inverse}{^{-1}}

\begin{document}
\author[I. Dimitrov]{Ivan Dimitrov}
\address{Department of Mathematics and Statistics, Queen's University}
\email{dimitrov@queensu.ca}

\author[C. Paquette]{Charles Paquette}
\address{Department of Mathematics and Computer Science, Royal Military
College of Canada}
\email{charles.paquette.math@gmail.com}

\author[D. Wehlau]{David Wehlau}
\address{Department of Mathematics and Computer Science, Royal Military
College of Canada}
\email{wehlau@rmc.ca}

\author[T. Xu]{Tianyuan Xu}
\address{Department of Mathematics and Statistics, Haverford College, Haverford PA, USA}
\email{txu2@haverford.edu}


\subjclass{Primary 20C07; Secondary 16S34, 16U40, 16D40}

\title{Idempotents in the group algebra of the infinite dihedral group}

\begin{abstract}
    We prove that over an algebraically closed field $\K$ of characteristic different from $2$, the group algebra $R=\K D_\infty$ of the infinite dihedral group $D_\infty$ has exactly six conjugacy classes of involutions (equivalently, of idempotents). This allows us to recover the fact that $R$ admits exactly four non-isomorphic indecomposable projective modules of the form $eR$ where $e$ is an idempotent, a result that was first established by Berman and Buz\'asi.
\end{abstract}

\maketitle

\section{Introduction}

In this short note, we consider an algebraically closed field $\K$ of characteristic different from
$2$ and study the involutions of the group algebra $R=\K D_\infty$ where $D_\infty$ is the infinite
dihedral group $D_\infty = \ip{a,b \mid a^2=b^2=1}$. Our goal is to give an elementary,
self-contained classification of conjugacy classes of involutions in $R$. Our main result is the following theorem. 

\begin{thm}
\label{thm:main}
    The algebra $R$ has exactly six conjugacy classes of involutions, namely, those of the elements $1, -1, \frac{1+a}{2},  \frac{1-a}{2},\frac{1+b}{2}$ and $\frac{1-b}{2}$.
\end{thm}

The group $D_\infty$ and the algebra $R$ are interesting for many reasons, some of which have a
representation-theoretic nature. For example, the group $D_\infty$ admits every group presented by
two involutions as a quotient, including every finite dihedral group and the infinite cyclic group,
therefore the representation theory of $R$ is closely related to that of such groups.  Also, the
algebra $R$ is isomorphic to the free product $\K C_2 * \K C_2$ of two copies of the semisimple algebra $\K C_2$, 
the group algebra
of the cyclic $C_2$ of order 2.  It is interesting to note that among all free
products of group algebras of finite cyclic groups, $R$ stands out as the only algebra
with infinite representation type that has a global bound on the dimensions of the finite
dimensional simple modules; see \cite[Proposition 5.4]{dpwx}.  More generally, among all free
products of finitely many nontrivial semisimple $\K$-algebras, $R$ also stands out as the only such
algebra of tame representation type; see \cite[Theorem 7.3]{usra}. Furthermore, the category of finite
dimensional modules of $R$ can be realized as a subcategory of the category of regular modules over
the $4$-subspace quiver $\tilde D_4$, so that the finite dimensional indecomposable modules are all 
uniserial and the collection of simple modules consists of four one-dimensional modules together with a
one-parameter family of simple modules of dimension 2; see \cite[Example 7.2]{usra}.

Finitely generated modules of the algebra $R$ have also been studied extensively. In \cite{BB}, Berman and Buz\'asi
consider the infinite cyclic subgroup $C$ of $D_\infty$ generated by the element $ab$ and restrict a
given $\K D_\infty$-module to a $\K C$-module, allowing them to exploit the well-understood
representation theory of $\K C$. Among other things, this fruitful approach yields the following
facts. Up to isomorphism, there are only four indecomposable infinite dimensional finitely generated
(right) $R$-modules, namely, those of the form $eR$ where $e$ is one of the four idempotents
$(1-a)/2, (1+a)/2, (1-b)/2, (1+b)/2$. Moreover, these modules are all projective, and they are the only indecomposable modules having no $\K
C$-torsion.  

The results mentioned above immediately imply that $R$ has at least four conjugacy classes of primitive idempotents, because if two
idempotents $e$ and $e'\in R$ are conjugate then the modules $eR$ and $eR'$ are isomorphic.
However, an isomorphism $eR\cong eR'$ for idempotents $e,e'$ does not guarantee $e$ and $e'$ are
conjugate (see \cite[Lemma 2.2]{idem}), so it is not obvious whether $R$ has exactly four conjugacy
classes of primitive idempotents. Theorem \ref{thm:main} answers this question affirmatively. Combining Proposition~\ref{noniso} with Theorem \ref{thm:main}, we prove the fact that $R$ has exactly four isomorphism classes of indecomposable projective modules of the form $eR$ where $e$ is an idempotent. We note that this fact follows from Theorem 6 of \cite{BB}, whose proof involves careful analysis of restrictions of $\K D_\infty$-modules to $\K C$, but the proof given here is shorter and more elementary.

\section{Preliminaries and main result}

Let $R=\K\ip{a,b\,\vert\, a^2=b^2=1}$. We want to classify the involutions of
$R$ up to conjugacy. Doing so is equivalent to classifying the idempotents of
$R$ up to conjugacy, because an element $r\in R$ is an idempotent if and only
if $2r-1$ is an involution, and two elements $r,r'\in R$ are conjugate if and
only if $2r-1$ and $2r'-1$ are conjugate.

Let $s=a, t=ab$ and let $\cala=\K[t,t\inverse]$ be the corresponding ring of Laurent polynomials. Observe that $R$ can also be
presented as $R=\K\ip{s,t:s^2=1,ts=st\inverse}$. In particular, $R$ is spanned by elements of form $t^n$
and $st^n$ where $n\in \Z$, so every element $u\in R$ can be written in the
form $f(t)+sg(t)$ for some $f,g\in
\cala$. It is straightforward to check that $R$ embeds into $M_{2\times 2}(\cala)$
via the injective homomorphism 
\begin{equation} \label{eq55}
  \iota: R\ra M_{2\times 2}(\cala), \quad f(t)+sg(t)\mapsto
  \begin{bmatrix}
    f(z) & g^*(z)\\
    g(z) & f^*(z)
  \end{bmatrix}
\end{equation}
where $*:\cala\ra\cala$ is the involutive map defined by $h^*(z)=h(1/z)$ for
all $h\in \cala$. We will henceforth often identify each 
$u\in R$ with its image $\iota(u)\in M_{2\times 2}(\cala)$. 

With the identification provided by \eqref{eq55}, Theorem~\ref{thm:main} can now be rephrased as follows.
\begin{thm} \label{th2.44}
  The involutions of $R$ fall into six conjugacy classes. These
  classes are represented by the elements 
\[
  1=
  \begin{bmatrix}
    1 & 0 \\
   0 &  1
  \end{bmatrix},
   s=
  \begin{bmatrix}
    0 & 1\\
    1 & 0\\
  \end{bmatrix},
   st=
  \begin{bmatrix}
    0 &  1/t\\
    t & 0
  \end{bmatrix},
\]
and their additive inverses $-1,-s,-st$; in other words, the classes are represented by the elements $1,-1,a,-a,b,-b$.
\end{thm}

Corresponding to the non-trivial involutions we have the following four
non-trivial idempotents: 
$$e=\frac{1-a}{2}, \quad e'=\frac{1-b}{2}, \quad 1-e=\frac{1+a}{2}, \quad 1-e'=\frac{1+b}{2}\ .$$
We will use this notation for these idempotents for the remainder of the paper.

By the idempotent-involution correspondence, the following is now immediate.

\begin{corollary}
    There are precisely four conjugacy classes of primitive idempotents in $R$, namely, the classes represented by $e, 1-e, e', 1-e'$. 
\end{corollary}

Next we study the projective modules of the form $xR$ where $x$ is an idempotent.  

\begin{Proposition} \label{noniso}
    There are exactly four pairwise non-isomorphic indecomposable projective modules of the form $xR$ where $x$ is an idempotent. They are the modules $eR, (1-e)R, e'R$ and $(1-e')R$ where $e$ and $e'$ are as in Corollary 3.
\end{Proposition}

\begin{proof}
For $\mu, \nu \in \{-1,1\}$, we can define a simple $R$ module $S_{\mu, \nu}$ of dimension one for which $a$ acts by multiplication by $\mu$ and $b$ acts by multiplication by $\nu$. The following hold.
\begin{enumerate}
    \item $\Hom_R(eR,-)$ vanishes on $S_{1,1}$ and $S_{1,-1}$;
    \item $\Hom_R((1-e)R,-)$ vanishes on $S_{-1,1}$ and $S_{-1,-1}$;
    \item $\Hom_R(e'R,-)$ vanishes on $S_{1,1}$ and $S_{-1,1}$; and
    \item $\Hom_R((1-e')R,-)$ vanishes on $S_{1,-1}$ and $S_{-1,-1}$.
\end{enumerate}
These facts immediately imply the desired result.
\end{proof}

Notice that from the equality
\begin{equation}\label{noKrullSchmidt}
eR \oplus (1-e)R = e'R \oplus (1-e')R,
\end{equation}
we have two distinct decompositions of the right $R$-module $R$ into a direct sum of indecomposable modules. In particular, the category of finitely generated (projective) right $R$-modules does not satisfy the Krull-Schmidt decomposition theorem. The relation \eqref{noKrullSchmidt} appears in \cite[Remark 2]{BB}. Moreover, \cite[Theorem 4]{BB} 
states that the Grothendieck group $K_0(R)$ of $R$ is a free abelian group of rank 3, 
and \eqref{noKrullSchmidt} is the only relation among its generators.

\section{Proof of the theorem}
Throughout this section, we use the embedding $\iota$ from (\ref{eq55}) to identify $R$ with its image in $M_{2 \times 2}(\mathcal{A})$.  Let 
\begin{equation}
\label{u eps theta}
        u_{\eps,\theta}=
        \begin{bmatrix}
          0 & \eps t^{-\theta}\\
          \eps t^\theta & 0
        \end{bmatrix}.
      \end{equation}
      The set $\{u_{\eps,\theta}: \eps\in
      \{1,-1\},\theta\in\{0,1\}\}$ coincides with the set
      $\{s,-s,st,-st\}$. To prove Theorem \ref{th2.44}, it suffices to prove that every element $u$ of $R$ such that
      $u^2 =1$ and $u \neq \pm 1$
      is conjugate to exactly one element of the set $\{u_{\eps,\theta}: \eps\in
      \{1,-1\},\theta\in\{0,1\}\}$.

Let $u \in R$ be as above. Since $u^2=1$, the minimal polynomial $m_x(u)$ of $u$
divides $x^2-1=(x-1)(x+1)$, therefore $m_x(u)$ equals $x-1, x+1$ or $x^2-1$. In
the first two cases we have $u=1$ and $u=-1$, respectively, so it suffices to
assume $m_{x}(u)=x^2-1$ and show that $u$ is conjugate to one of the four elements $\pm s, \pm
st$. 

Let $Q=\K(t)$ denote the field of fractions of $\cala$. 
Since $u\in M_{2\times 2}(Q)$ and $m_x(u)=x^2-1$, the characteristic polynomial of $u$ equals $x^2-1$,
therefore $u$ has trace 0 and determinant $-1$.
It follows that $u$ is of the form 
\[
u=
\begin{bmatrix}
  f & g^*\\
  g & f^*
\end{bmatrix}
\]
where $f,g\in \cala$, $f^*=-f$ and $ff^{*}-gg^*=-1$. 

Recall that the ring $\cala$ is
a unique factorization domain with units given by $\cala^\times=\{\lambda t^m: \lambda\in
\C^\times, m\in \Z\}$. Consider the unique prime factorization of $1+f$ given by 
\begin{equation} 
  \label{1plusf}
  1+f=\delta t^m \prod_{i=1}^n p_i,
\end{equation}
where $\delta\in \C^\times, m\in \Z$ and $p_i$ is a prime element in $\cala$
for each $i\in [n]$. 
Then as $f^*=-f$, we have 
\begin{equation} 
  \label{1minusf}
  1-f=1+f^*=(1+f)^*=\delta t^{-m}\prod_{i=1}^n p_i^*
\end{equation}
and hence 
\begin{equation}
  \label{1-f^2}
  1+ff^*=1-f^2=(1+f)(1-f)=\delta^2 \prod_{i=1}^n p_ip_i^*.
\end{equation}
Since $gg^*=1+ff^*$, it further follows that there is a subset $I$ of
$[n]$ such that 
\begin{equation}
  \label{gg*}
  g=\gamma t^l \prod_{i\in I}p_i\prod_{i\in I^c}p_i^*, \quad
  g^*=\gamma t^{-l} \prod_{i\in I}p_i^*\prod_{i\in I^c}p_i
\end{equation}
where $\gamma\in \C^\times, l\in \Z$ and $I^c=[n]\setminus I$.
Note that
$\delta^2=\gamma^2$ by equations \eqref{1-f^2} and \eqref{gg*}, so
$\gamma=\varepsilon\delta$ and $\delta=\varepsilon\gamma$ with
 $\varepsilon \in \{1, -1\}$. Let $\theta$ be the unique integer in $\{0,1\}$ with the same parity as
$l+m$. Note that both the parameters $\varepsilon$ and $\theta$ are uniquely
determined by $f$ and $g$ and hence by $u$.

Next we show that $u$ is conjugate to $u_{\eps,\theta}$ by 
constructing $\nu\in R^\times$ such that $\nu\inverse u\nu=u_{\eps,\theta}$. To define $\nu$, let 
\begin{equation}
  g_1=\varepsilon  t^{\frac{l+m+\theta}{2}}\prod_{i\in I}p_i, \quad
  g_2= \delta t^{\frac{l-m-\theta}{2}}\prod_{i\in I^c}p_i^*,
  \label{g1g2}
\end{equation}
let 
\begin{equation} 
  a=\dfrac{\eps t^\theta g_2+g_1}{2g}, \quad b=\dfrac{\eps t^{\theta}
  g_2-g_1}{2g}, 
  \label{ab def}
\end{equation}
and let 
\begin{equation}
  \label{nu def}
\nu=
\begin{bmatrix}
  \eps(a+bf) & bt^{-\theta}+at^{-\theta}f\\
  \eps bg & at^{-\theta}g
\end{bmatrix}\in M_{2\times 2}(Q).
\end{equation}

With these, we can now establish the following.

\begin{prop} \label{PropositionInvolutions}
  The element $\nu\in M_{2\times 2}(Q)$ satisfies the following properties.
  \begin{enumerate}
    \item $\nu\in \iota(R)$.
    \item $\nu$ is invertible in $M_{2\times 2}(Q)$, and its inverse
      $\nu\inverse$ in $M_{2\times 2}(Q)$ is also in
      $\iota(R)$.
    \item $\nu\inverse u\nu=u_{\eps,\theta}$.
  \end{enumerate}
\end{prop}
\begin{proof}
  (1) To prove $\nu\in\iota(R)$ it suffices to show that $at^{-\theta}g\in \cala,
  (at^{-\theta}g)^*=\eps(a+bf), \eps bg\in \cala$
 and
  $(\eps bg)^*=bt^{-\theta}+at^{-\theta}f$.
Note that it follows from Equation (\ref{ab def}) that
\[
  at^{-\theta}g=\frac{\eps g_2 +t^{-\theta}
g_1}{2},
\]
therefore $at^{-\theta}g\in \cala$
and \[
  (at^{-\theta}g)^*=\frac{\eps g_2^*+t^{\theta}g_1^*}{2}.
\]
Also, from equations (\ref{1plusf}) and (\ref{1minusf}), it follows that
\begin{equation}
  \label{quotients}
  \frac{1+f}{g_1}= \gamma t^{\frac{m-l-\theta}{2}}\prod_{i\in I^c}p_i=\eps
  t^{-\theta}g_2^*,\quad
  \frac{1-f}{g_2}=t^{\frac{-l-m+\theta}{2}}\prod_{i\in I}p_i^*=\eps
  t^{\theta}g_1^*.
\end{equation}
Combining (\ref{ab def}), \eqref{quotients}, and the fact that $g_1g_2=g$, we get  
\[
  \eps(a+bf)=\frac{t^{\theta}g_2+\eps g_1+ t^\theta g_2 f-\eps g_1f}{2g}
  =\dfrac{t^\theta \frac{1+f}{g_1}+\eps \frac{1-f}{g_2}}{2}=\frac{\eps
    g_2^*+t^{\theta}g_1^*}{2}. 
\]
 Therefore, we have 
\[
(at^{-\theta}g)^*=\eps(a+bf).
\]
Using (\ref{ab def}) again yields
\[
  \eps bg=\dfrac{t^{\theta}g_2-\eps g_1}{2}\in \cala .
\]
and using (\ref{ab def}) with (\ref{quotients}) yields
\begin{eqnarray*}
bt^{-\theta} + at^{-\theta}f & =& \dfrac{\eps g_2 - t^{-\theta}g_1 + \eps f g_2 + t^{-\theta}f g_1}{2g}\\ &=& \dfrac{\eps g_2 (\frac{1+f}{g}) - t^{-\theta}g_1(\frac{1-f}{g})}{2}\\ & = & 
\dfrac{\eps (\frac{1+f}{g_1}) - t^{-\theta}(\frac{1-f}{g_2})}{2}\\ & = &\dfrac{t^{-\theta}g_2^*-\eps
g_1^*}{2} \\
& = & (\eps b g)^*,
\end{eqnarray*}
which completes our proof of the fact that $\nu\in \iota(R)$.

\medskip
  
(2) A straightforward calculation shows that 
\[
  \det \nu=\eps t^{-\theta}g (a^2-b^2)=\eps
  t^{-\theta}g(a+b)(a-b).
\]
Since $g_1g_2=g$, it follows from \eqref{ab def} that
\[
  a+b=\eps t^\theta g_1\inverse, \quad a-b=g_2\inverse
\]
and hence $\det\nu=1$. It further follows
that the inverse $\nu\inverse$ of $\nu$ in $M_{2\times 2}(Q)$ is given by
\[
  \nu\inverse=\frac{1}{\det \nu}
  \begin{bmatrix}
    at^{-\theta}g &-(bt^{-\theta}+at^{-\theta}f)\\
     -\eps bg &  \eps(a+bf) 
  \end{bmatrix}
=
  \begin{bmatrix}
    at^{-\theta}g &-(bt^{-\theta}+at^{-\theta}f)\\
     -\eps bg &  \eps(a+bf) 
  \end{bmatrix}
\]
from which it is clear that $\nu\inverse\in \iota(R)$, because $\nu\in \iota(R)$. 

\medskip

  (3) Using the facts that $f^*=-f$ and $ff^*-gg^*=-1$, it is easy to check that $u\nu=\nu u_{\eps,\theta}$, therefore $\nu\inverse
  u\nu=u_{\eps,\theta}$.
\end{proof}

By Part (3) of Proposition \ref{PropositionInvolutions}, we may now conclude that $\{\pm 1, \pm s, \pm
st\}$ forms a complete list of representatives of conjugacy classes of
involutions in $R$. Equivalently, the corresponding four primitive idempotents $e, 1-e, e', 1-e'$ form a complete list of representatives of conjugacy classes of primitive idempotents in $R$. As observed in Proposition \ref{noniso}, the corresponding projective modules are pairwise non-isomorphic, therefore these four idempotents are pairwise non-conjugate. This completes the proof.

\medskip
\noindent
{\bf {Acknowledgement.}}
 I.D., C.P., and D.W. were partially supported by the
Natural Sciences and Engineering Research Council of Canada. In addition, C.P. and D.W.  
were partially supported by the Canadian Defence
Academy Research Programme.

\bibliographystyle{alpha}\nocite{*}
\bibliography{inv.bib}
\end{document}